\documentclass[11pt]{amsart}
\usepackage[final]{microtype}
\usepackage[a4paper,text={160mm,254mm},centering]{geometry}
\usepackage{lmodern,hyperref,amssymb,mathtools}

\usepackage[cal=euler]{mathalpha}

\renewcommand{\leq}{\leqslant}
\renewcommand{\geq}{\geqslant}
\renewcommand{\subset}{\subseteq}

\newcommand{\Kl}{\mathrm{Kl}}
\newcommand{\IKl}{\mathrm{IKl}}
\newcommand{\Nrd}{\mathrm{Nrd}}
\newcommand{\Trd}{\mathrm{Trd}}
\newcommand{\Norm}{\mathrm{Norm}}
\newcommand{\Trace}{\mathrm{Tr}}
\newcommand{\chitriv}{\mathbb{1}}

\numberwithin{equation}{section}

\theoremstyle{remark}
\newtheorem{remark}[equation]{Remark}

\theoremstyle{plain}
\newtheorem{theorem}[equation]{Theorem}
\newtheorem{lemma}[equation]{Lemma}
\newtheorem{proposition}[equation]{Proposition}
\newtheorem{corollary}[equation]{Corollary}

\theoremstyle{definition}
\newtheorem{definition}[equation]{Definition}

\title{Exotic and inverted Kloosterman sums over semisimple algebras}

\author{Daqing Wan}
\address{Center for Discrete Mathematics,
  Chongqing University; and College of Mathematics and Statistics,
  Chongqing University, Chongqing 401331, China}
\email{dwan@math.uci.edu}

\author{Dingxin Zhang}
\address{Center for Mathematics and Interdisciplinary Sciences, Fudan
  University; and  Shanghai Institute for Mathematics and
  Interdisciplinary Sciences (SIMIS), Shanghai 200433, China}
\email{dingxinzhang@fudan.edu.cn}

\subjclass{Primary: 11L40; Secondary: 11M38}

\begin{document}

\begin{abstract}
We introduce exotic Kloosterman sums and exotic inverted Kloosterman sums attached
to non-commutative finite-dimensional semisimple algebras over a finite field
\(\mathbb{F}_q\), and prove their reduction formulae to exotic
Kloosterman and exotic inverted Kloosterman sums over commutative étale \(\mathbb{F}_q\)
algebras. We then obtain square-root estimates for these sums; for
inverted sums an explicit correction term may appear.
\end{abstract}

\maketitle

% \tableofcontents

\section{Introduction}

Let \(\mathbb{F}_q\) be a finite field, and fix a nontrivial additive
character \(\psi\colon\mathbb{F}_q\to\mathbb{C}^{\times}\).  Recall
that the classical twisted hyper-Kloosterman sum is
\[
\Kl(a;\chi_1,\ldots,\chi_m)
=
\sum_{\substack{x_1,\ldots,x_m\in\mathbb{F}_q^{\times}\\
    x_1\cdots x_m=a}}
\chi_1(x_1)\cdots\chi_m(x_m)\psi(x_1+\cdots+x_m),
\qquad a\in\mathbb{F}_q^\times.
\]
Classical Kloosterman sums are basic objects in the theory of exponential
sums.
A variant of them, called \emph{inverted Kloosterman sums}, appears in
the theory of Ramanujan graphs:
\begin{equation*}
\IKl(a; \chi_1,\ldots,\chi_m)
=
\sum_{\substack{x_1,\ldots,x_m\in\mathbb{F}_q^{\times}\\
    x_1\cdots x_m=a \\ x_1 + \cdots + x_m \neq 0}}
\chi_1(x_1)\cdots\chi_m(x_m)\psi\left( \frac{1}{x_1+\cdots+x_m} \right),
\qquad a\in\mathbb{F}_q^\times.
\end{equation*}
Deligne \cite{deligne_sommes-trig} and Katz
\cite{katz_gauss-sums-kloosterman-sums-and-monodromy} introduced
``exotic'' generalizations of Kloosterman sums, namely Kloosterman
sums over finite étale \(\mathbb{F}_q\)-algebras.  Already in Katz's
note on inverted Kloosterman sums, their exotic variants were
considered \cite{katz_note-exponential-sums}.  These exotic inverted
Kloosterman sums, together with related norm-trace counting problems,
have recently been studied by many authors
\cite{lin-wan_inverted-kloosterman,
  lin-wan_trace-norm-etale,fu-wan_exotic-inverted-kloosterman,
  wan_normtrace-kloosterman-sums-finite-semisimple-algebras}.

A related direction concerns exponential sums over matrix groups, including
matrix Kloosterman sums in the sense of
Erdélyi--Tóth and their collaborators
\cite{erdelyi-toth_matrix-kloosterman,
  erdelyi-sawin-toth_purity-locus,
  erdelyi-toth-zabradi_matrix-prime-powers}, work of
Zelingher~\cite{zelingher_matrix-kloosterman-hall-littlewood}, and
hypergeometric exponential sums over reductive groups in work of
Fu--Li \cite{fu-li_hypergeometric-reductive-groups}.  In this paper,
we extend the two constructions above from the étale case to
finite-dimensional semisimple \(\mathbb{F}_q\)-algebras.  Thus the sums
below may also be viewed as exponential sums over spaces of matrices.

\begin{definition}
Let \(M\) be a finite-dimensional semisimple algebra over
\(\mathbb{F}_q\).  Denote by \(M^{\times}\) the group of invertible
elements of \(M\), by \(\Trd\colon M \to \mathbb{F}_q\) the
\emph{reduced trace}, and by \(\Nrd\colon M \to \mathbb{F}_q\)
the \emph{reduced norm}.  Let
\(\chi\colon M^{\times} \to \mathbb{C}^{\times}\) be a multiplicative
character.  Then for \(a\in\mathbb{F}_q^\times\), we define the
\emph{Kloosterman sum} of the semisimple \(\mathbb{F}_q\)-algebra
\(M\) associated to the multiplicative character \(\chi\) by
\begin{equation*}
\Kl_M(a;\chi)
=
\sum_{\substack{x\in M^\times\\ \Nrd(x)=a}}
\chi(x)\psi(\Trd(x)).
\end{equation*}
\end{definition}

\begin{remark}
By Wedderburn's theorem and the triviality of the Brauer group of a
finite field, any finite-dimensional semisimple
\(\mathbb{F}_q\)-algebra \(M\) is isomorphic to
\begin{equation}\label{eq:M-decomp}
M\simeq
M_{n_1}(\mathbb{F}_{q^{d_1}})
\times\cdots\times
M_{n_k}(\mathbb{F}_{q^{d_k}}).
\end{equation}
We shall denote by
\[
B=\mathbb{F}_{q^{d_1}}\times\cdots\times\mathbb{F}_{q^{d_k}}
\]
the center of \(M\).
\end{remark}

In this paper, we restrict to \emph{determinant-type characters}:
under the identification \eqref{eq:M-decomp}, we say
\(\chi=(\chi_1,\ldots,\chi_k)\) is of determinant type if it comes
from a multiplicative character of \(B^{\times}\), pulled back along
the determinant map
\[
M^\times\longrightarrow B^\times,\qquad
(A_1,\ldots,A_k)\longmapsto (\det A_1,\ldots,\det A_k).
\]
Thus, for
\(x=(A_1,\ldots,A_k)\in M^\times\),
\[
\chi(x)=\prod_i \chi_i(\det A_i).
\]

Under the identification \eqref{eq:M-decomp}, the reduced norm and
reduced trace of \(M\) are given by the following formulas:
\[
\begin{aligned}
  \Nrd(A_1,\ldots,A_k)
  &=
    \prod_{i=1}^k
    \Norm_{\mathbb{F}_{q^{d_i}}/\mathbb{F}_q}(\det A_i),\\
  \Trd(A_1,\ldots,A_k)
  &=
    \sum_{i=1}^k
    \Trace_{\mathbb{F}_{q^{d_i}}/\mathbb{F}_q}(\operatorname{Tr} A_i).
\end{aligned}
\]

With this notation, we set
\begin{equation}
\label{eq:reduced-algebra}
B^{\prime}=
\mathbb{F}_{q^{d_1}}^{n_1}
\times\cdots\times
\mathbb{F}_{q^{d_k}}^{n_k}.
\end{equation}
Concretely, an element of \((B^{\prime})^\times\) is a tuple
\((x_{ij})_{1\leq i\leq k,\,1\leq j\leq n_i}\) with every coordinate
nonzero.  This is a finite \'etale \(\mathbb{F}_q\)-algebra.  Let
\(\eta\) be the character of \((B^{\prime})^\times\) given by
\begin{equation}
\label{eq:induced-char}
\eta((x_{ij}))=\prod_{i=1}^k\prod_{j=1}^{n_i}\chi_i(x_{ij}).
\end{equation}
Our first main result is the following reduction formula.  It reduces
the Kloosterman sum over the noncommutative algebra \(M\) to the
commutative étale algebra \(B^{\prime}\), where Katz's estimates apply.

\begin{theorem}\label{thm:reduction}
With notation as above, and with \(\chi\) of determinant type, for every
\(a\in\mathbb{F}_q^\times\), we have
\begin{equation*}
\Kl_M(a;\chi)
=
q^{N} \cdot \Kl_{B^{\prime}}(a;\eta),
\qquad
N=\sum_{i=1}^k d_i\binom{n_i}{2}.
\end{equation*}
\end{theorem}

For the trivial multiplicative character and \(M=M_n(\mathbb{F}_q)\),
Theorem~\ref{thm:reduction} recovers Kim's formula \cite{kim}
\[
\Kl_{M_n(\mathbb{F}_q)}(a;\chitriv)
= q^{\binom{n}{2}}\,\Kl_n(a).
\]
For general semisimple \(M\), it reduces the problem to Katz's
commutative finite-\'etale case.  Combining with Katz's square-root
cancellation estimate for exotic Kloosterman sums on \'etale
\(\mathbb{F}_q\)-algebras gives the following square-root cancellation
estimate.

\begin{corollary}\label{cor:bound}
We have
\[
|\Kl_M(a;\chi)|
\leq
(\dim B^{\prime})\, q^{\frac{\dim M-1}2}
=
(d_1n_1 + \cdots + d_k n_k)
q^{\frac{d_1 n_1^2 + \cdots + d_k n_k^2 - 1}2}.
\]
\end{corollary}

\begin{proof}
Since \(B^{\prime}\) is a commutative étale algebra over
\(\mathbb{F}_q\), the sum \(\Kl_{B^{\prime}}(a;\eta)\) is the
so-called \emph{exotic Kloosterman sum} studied by
\cite{katz_gauss-sums-kloosterman-sums-and-monodromy}.  In
\cite[Theorem~8.8.5]{katz_gauss-sums-kloosterman-sums-and-monodromy},
it is shown that for any étale \(\mathbb{F}_q\)-algebra \(E\), any
multiplicative character
\(\theta\colon E^{\times}\to \mathbb{C}^{\times}\), and any
\(a \in \mathbb{F}_q^{\times}\), we have
\(|\Kl_E(a;\theta)| \leq (\dim_{\mathbb{F}_q} E)
q^{\frac{\dim_{\mathbb{F}_q} E - 1}2}\).  Since
\(\dim_{\mathbb{F}_q} M = \dim_{\mathbb{F}_q} B^{\prime} + 2N\),
the corollary follows from this estimate applied to \(B^{\prime}\) and
the reduction formula,
Theorem~\ref{thm:reduction}.
\end{proof}

\begin{remark}
The determinant-type hypothesis is automatic except for one small
exception.  For \(G=\mathrm{GL}_n(F)\), every multiplicative character
\(G\to\mathbb{C}^{\times}\) factors through the determinant unless
\(G=\mathrm{GL}_2(\mathbb{F}_2)\).  Equivalently, the determinant gives
the abelianization of \(\mathrm{GL}_n(F)\), except that
\(\mathrm{GL}_2(\mathbb{F}_2)\simeq S_3\).  In this exceptional case
\(\mathbb{F}_2^{\times}\) is trivial, while the sign character of
\(S_3\) gives a non-determinant-type character.  Consequently, for a
product of matrix algebras, the only multiplicative characters not of
determinant type are obtained by multiplying by sign characters on the
factors \(M_2(\mathbb{F}_2)\).  Thus, if \(M\) has no
\(M_2(\mathbb{F}_2)\)-factor, every multiplicative character of
\(M^{\times}\) is of determinant type.
\end{remark}

Next, let us turn to the inverted sums.

\begin{definition}
Let \(M\) be a finite-dimensional semisimple algebra over
\(\mathbb{F}_q\), and let
\(\chi\colon M^{\times} \to \mathbb{C}^{\times}\) be a multiplicative
character.  For \(a\in\mathbb{F}_q^\times\), we define the
\emph{inverted Kloosterman sum} over \(M\), associated to the
multiplicative character \(\chi\), by
\begin{equation*}
\IKl_M(a;\chi)
=
\sum_{\substack{x\in M^\times\\ \Nrd(x)=a \\ \Trd(x) \neq 0}}
\chi(x)\,\psi\left( \frac{1}{\Trd(x)} \right).
\end{equation*}
\end{definition}

We shall also prove the corresponding reduction formula for inverted
Kloosterman sums, with an explicit correction term.  We recall the
definition of the algebra
\(B^{\prime}\) \eqref{eq:reduced-algebra} and the multiplicative
character \(\eta\) on \((B^{\prime})^{\times}\)
\eqref{eq:induced-char}.

\begin{theorem}\label{thm:intro-inverted-reduction}
Assume that \(\chi\) is of determinant type, and let
\(a\in\mathbb{F}_q^\times\).  Denote by
\(\nu\colon (B^{\prime})^{\times} \to \mathbb{F}_q^{\times}\) the norm
map.  Thus
\[
\nu((x_{ij}))=
\prod_i
\Norm_{\mathbb{F}_{q^{d_i}}/\mathbb{F}_q}(x_{i1}\cdots x_{in_i}).
\]
Put
\[
N=\sum_{i=1}^k d_i\binom{n_i}{2},
\]
and let
\(\mathcal{B}_M\) be the product of the groups of invertible upper
triangular matrices in the factors
\(\mathrm{GL}_{n_i}(\mathbb{F}_{q^{d_i}})\).  Then the following
assertions hold.
\begin{enumerate}
\item If \(\eta\) is nontrivial on \(\ker(\nu)\), then
\[
\IKl_M(a;\chi)=q^N\, \IKl_{B^{\prime}}(a;\eta).
\]
\item If \(\eta\) is trivial on \(\ker(\nu)\), equivalently if
\(\eta = \rho \circ \nu\) for some multiplicative character
\(\rho\) of \(\mathbb{F}_q^\times\), then
\[
\IKl_M(a;\chi)
= q^N\,\IKl_{B^{\prime}}(a;\eta)
-\rho(a)\frac{\#M^\times-\#\mathcal{B}_M}{q(q-1)}.
\]
Equivalently, since \(\#\mathcal{B}_M = q^N\# (B^{\prime})^{\times}\),
we have
\[
\IKl_M(a;\chi) + \rho(a)\frac{\#M^\times}{q(q-1)}
= q^N\,\left(\IKl_{B^{\prime}}(a;\eta)
+\rho(a)\frac{\# ({B}^{\prime})^{\times}}{q(q-1)}\right).
\]
\end{enumerate}
\end{theorem}

Combining the reduction theorem with the estimates in
\cite[Theorems~1.1--1.2]{fu-wan_exotic-inverted-kloosterman} for
\(\IKl_{B^{\prime}}(a;\eta)\) gives the corresponding square-root
estimate.

\begin{theorem}
Let
\[
M = M_{n_1}(\mathbb{F}_{q^{d_1}}) \times \dots \times
M_{n_k}(\mathbb{F}_{q^{d_k}}),
\qquad
n = \sum d_i n_i^2 \geq 2,
\qquad
m = \sum d_i n_i.
\]
In the notation of Theorem~\ref{thm:intro-inverted-reduction}, the
following estimates hold.
\begin{enumerate}
\item Suppose that \(p \nmid m\).
\begin{itemize}
\item If \(\eta\) is trivial on \(\ker(\nu)\), then
\[
\left| {\rm IK}_M(a, \chi) + \rho(a)\frac{\# M^\times}{q(q-1)} \right|
\leq 2m \cdot q^{\frac{n-1}{2}}.
\]
\item If \(\eta\) is nontrivial on \(\ker(\nu)\), then
\[
|{\rm IK}_M(a, \chi)| \leq 2m \cdot q^{\frac{n-1}{2}}.
\]
\end{itemize}

\item Suppose that \(p \mid m\), and that either \(p > 2\) or
\(m \not\equiv 2 \pmod 4\).
\begin{itemize}
\item If \(\eta\) is trivial on \(\ker(\nu)\), then
\[
\left| {\rm IK}_M(a, \chi) + \rho(a) \frac{\# M^\times}{q(q-1)} \right|
\leq m \cdot q^{\frac{n-1}{2}}.
\]
\item If \(\eta\) is nontrivial on \(\ker(\nu)\), then
\[
|{\rm IK}_M(a, \chi)| \leq m \cdot q^{\frac{n-1}{2}}.
\]
\end{itemize}
\end{enumerate}
\end{theorem}

In the wild case \(p \mid m\), the estimate is stronger: the
coefficient is \(m\), not \(2m\), as in the tame case \(p\nmid m\).

\section{Proof of Theorem~\ref*{thm:reduction}}

We shall begin by considering the case where \(k = 1\) and no
multiplicative character is present.  The following lemma is due to Kim
\cite{kim}, but we include the
proof for completeness.  The presentation is slightly different from
Kim's proof, which uses an induction procedure.

\begin{lemma}\label{lem:matrix-reduction}
Let \(F=\mathbb{F}_{q^d}\), let
\(\psi_F=\psi\circ\Trace_{F/\mathbb{F}_q}\), and let
\(Q=q^d\).  For \(e\geq 1\) and \(y\in F^\times\),
\begin{equation}\label{eq:matrix-reduction}
\sum_{\substack{A\in\mathrm{GL}_e(F)\\ \det A=y}}
\psi_F(\operatorname{Tr}A)
=
Q^{\binom{e}{2}}
\sum_{\substack{x_1,\ldots,x_e\in F^\times\\  x_1\cdots x_e=y}}
\psi_F(x_1+\cdots+x_e).
\end{equation}
\end{lemma}

\begin{proof}
Let \(\mathcal{B}\subset\mathrm{GL}_e\) be the group of invertible upper
triangular matrices, and let \(\mathcal{U}\subset \mathcal{B}\) be the subgroup of upper
triangular matrices whose diagonal entries are all \(1\).  Thus an
element of \(\mathcal{U}\) has the form
\[
1+\sum_{r<s}u_{rs}E_{rs}.
\]
For \(w\in S_e\), choose a monomial matrix \(\dot w\), meaning a
matrix with exactly one nonzero entry in each row and each column,
whose nonzero entries occur in the same positions as the permutation
matrix of \(w\).  We choose these nonzero entries so that
\(\det\dot w=1\), and take \(\dot 1=1\).  Concretely, one may start
with the usual permutation matrix and rescale one of its nonzero
entries by the inverse of its determinant.

In this matrix setting, the Bruhat decomposition says that
\[
\mathrm{GL}_e(F)=\bigsqcup_{w\in S_e} \mathcal{U}_w(F)\dot w\mathcal{B}(F),
\]
where \(\mathcal{U}_w\) is a subset of \(\mathcal{U}\), described as follows.  Write \(w\)
in one-line notation
\[
w=(w(1),w(2),\ldots,w(e)).
\]
An \emph{inversion} of \(w\) is a pair of positions \(a<b\) for which
\(w(a)>w(b)\).  For each such inversion, allow one free
upper-triangular entry in position \((w(b),w(a))\).  These are indeed
upper-triangular positions, since \(w(b)<w(a)\).  All other
upper-triangular entries are set equal to \(0\).  Thus
\[
\mathcal{U}_w=
\left\{
1+\sum_{\substack{a<b\\w(a)>w(b)}} u_{ab}E_{w(b),w(a)}
: u_{ab}\in F
\right\}.
\]
For example, if \(e=3\) and \(w=(2,3,1)\), the inversions are
\((1,3)\) and \((2,3)\), and
\[
\mathcal{U}_w(F)=
\left\{
\begin{pmatrix}
  1&u&v\\0&1&0\\0&0&1
\end{pmatrix}
: u,v\in F
\right\}.
\]
The number of free coordinates in \(\mathcal{U}_w\) is therefore the number of
inversions of \(w\), denoted \(\ell(w)\), and
\(\#\mathcal{U}_w(F)=Q^{\ell(w)}\).

Denote the left-hand side of \eqref{eq:matrix-reduction} by \(R_y\).
The disjoint union above gives the cell decomposition
\[
R_y=\sum_{w\in S_e} R_{y,w},
\]
where each summand is
\[
R_{y,w}
=
\sum_{\substack{u\in \mathcal{U}_w(F),\, b\in \mathcal{B}(F)\\
    \det(u\dot wb)=y}}
\psi_F(\operatorname{Tr}(u\dot wb)).
\]
Fix \(w\in S_e\).  Since \(\det\dot w=1\), the condition \(\det(u\dot wb)=y\) is simply
\(\det b=y\).  Also
\(\operatorname{Tr}(u\dot wb)=\operatorname{Tr}(b(u\dot w))\), and for
each fixed \(u\), the change of variables \(c=bu\) preserves \(\mathcal{B}(F)\)
and the determinant.  Hence
\[
R_{y,w}=\#\mathcal{U}_w(F)\,S_y(\dot w),
\]
where, for \(X\in\mathrm{GL}_e(F)\) and \(\alpha\in F^\times\), we put
\[
S_\alpha(X)=
\sum_{\substack{b\in \mathcal{B}(F)\\ \det b=\alpha}}
\psi_F(\operatorname{Tr}(bX)).
\]

To compute \(S_y(\dot w)\), expand the trace.  If \(X=(X_{rs})\) and \(b=(b_{rs})\), then
\[
\operatorname{Tr}(bX)
=
\sum_r b_{rr}X_{rr}
+
\sum_{r<s} b_{rs}X_{sr}.
\]
This is the key point: each entry below the diagonal of \(X\) appears
as the coefficient of one freely varying upper-triangular entry of
\(b\).

Now suppose \(w\neq 1\).  Then \(\dot w\) has a nonzero entry below the
diagonal, say in position \((s,r)\) with \(s>r\).  In the trace
formula above, the variable \(b_{rs}\) therefore occurs with
coefficient \(X_{sr}\neq 0\).  Summing over \(b_{rs}\in F\), we get
\[
\sum_{b_{rs}\in F}\psi_F(X_{sr}\,b_{rs})
=
0
\]
by linear independence of characters.
Hence \(S_y(\dot w)=0\) and \(R_{y,w}=0\) for every \(w\neq 1\).

It remains to treat the identity cell \(w=1\).  Here \(\dot 1=1\), so no upper off-diagonal variable
appears in \(\operatorname{Tr}(b)\).  Only the diagonal entries of
\(b\) remain, and therefore
\[
S_y(1)=
Q^{\binom{e}{2}}
\sum_{\substack{x_1,\ldots,x_e\in F^\times\\
    x_1\cdots x_e=y}}
\psi_F(x_1+\cdots+x_e).
\]
Since \(\#\mathcal{U}_1(F)=1\), the cell \(w=1\) contributes
\(S_y(1)\), which is exactly the right-hand side of
\eqref{eq:matrix-reduction}.  All other cells contribute zero.
\end{proof}

\begin{proof}[Proof of Theorem~\ref{thm:reduction}]
Write \(F_i=\mathbb{F}_{q^{d_i}}\),
\(\psi_i=\psi\circ\Trace_{F_i/\mathbb{F}_q}\), and
\(\Norm_i=\Norm_{F_i/\mathbb{F}_q}\).  Grouping first by the
componentwise determinants \(y_i=\det A_i\), we obtain
\[
\Kl_M(a;\chi)
=
\sum_{\substack{y_i\in F_i^\times\\
    \prod_i\Norm_i(y_i)=a}}
\prod_i\chi_i(y_i)\,R_i(y_i),
\]
where
\[
R_i(y_i)
=
\sum_{\substack{A_i\in\mathrm{GL}_{n_i}(F_i)\\
    \det A_i=y_i}}
\psi_i(\operatorname{Tr}A_i).
\]
Applying Lemma~\ref{lem:matrix-reduction} to each factor gives
\[
R_i(y_i)
=
q^{d_i\binom{n_i}{2}}
\sum_{\substack{x_{i1},\ldots,x_{in_i}\in F_i^\times\\
    x_{i1}\cdots x_{in_i}=y_i}}
\psi_i(x_{i1}+\cdots+x_{in_i}).
\]
Substituting and expanding the product over \(i\) yields
\[
\Kl_M(a;\chi)
=
q^{\sum_i d_i\binom{n_i}{2}}
\sum_{\substack{x_{ij}\in F_i^\times\\
    \prod_i \Norm_i(x_{i1}\cdots x_{in_i})=a}}
\prod_{i,j}\chi_i(x_{ij})
\psi\!\left(
\sum_i\Trace_{F_i/\mathbb{F}_q}
(x_{i1}+\cdots+x_{in_i})
\right).
\]
The remaining sum is precisely \(\Kl_{B^{\prime}}(a;\eta)\).
\end{proof}

\section{Proof of Theorem~\ref*{thm:intro-inverted-reduction}}

We view \((B^{\prime})^\times\) as the group of invertible diagonal matrices in
\(\prod_i\mathrm{GL}_{n_i}(\mathbb{F}_{q^{d_i}})\), and let
\(\nu\colon (B^{\prime})^\times\to\mathbb{F}_q^\times\) be the norm map.  For
each factor, let \(S_{n_i}\) be the permutation group on
\(\{1,\ldots,n_i\}\), and set
\[
W=S_{n_1}\times\cdots\times S_{n_k}.
\]
For \(w=(w_1,\ldots,w_k)\in W\), write
\[
\ell_d(w)=\sum_i d_i\,\ell(w_i),
\]
where \(\ell(w_i)\) is the inversion number of \(w_i\) in one-line
notation, as in Lemma~\ref{lem:matrix-reduction}.
Put
\[
N=\sum_i d_i\binom{n_i}{2}.
\]

\begin{lemma}\label{lem:inverted-affine}
Let \(L\colon\mathbb{A}^r(\mathbb{F}_q)\to\mathbb{F}_q\) be a
nonconstant affine-linear function.  Then
\[
\sum_{\substack{z\in\mathbb{A}^r(\mathbb{F}_q)\\L(z)\neq 0}}
\psi(L(z)^{-1})
=
-q^{r-1}.
\]
This is the inverted-character analogue of the cancellation identity
\(\sum_{z\in\mathbb{F}_q}\psi(cz)=0\) for \(c\neq 0\).
\end{lemma}

\begin{proof}
Every value of \(L\) occurs exactly \(q^{r-1}\) times.  Hence the
sum equals
\[
q^{r-1}\sum_{u\in\mathbb{F}_q^\times}\psi(u^{-1})
=
q^{r-1}\sum_{u\in\mathbb{F}_q^\times}\psi(u)
=
-q^{r-1}.
\]
\end{proof}

\begin{proposition}\label{prop:inverted-master}
With notation as above, for every \(a\in\mathbb{F}_q^\times\),
\begin{equation}\label{eq:inverted-master}
\IKl_M(a;\chi)
=
q^N\, \IKl_{B^{\prime}}(a;\eta)
-
\left(\sum_{\substack{w\in W\\ w\neq 1}} q^{N-1+\ell_d(w)}\right)
\sum_{\substack{t\in (B^{\prime})^\times\\ \nu(t)=a}} \eta(t).
\end{equation}
The sum over \(w\neq 1\) is interpreted as zero if \(W=\{1\}\).
\end{proposition}

\begin{proof}
Write \(F_i=\mathbb{F}_{q^{d_i}}\).  In the \(i\)-th factor
\(\mathrm{GL}_{n_i}(F_i)\), let \(\mathcal{B}_i\) be the group of
invertible upper triangular matrices, and let \(\mathcal{U}_i\) be the
subgroup of upper triangular matrices with diagonal entries all equal
to \(1\).

For each \(w_i\in S_{n_i}\), choose a determinant-one monomial
representative \(\dot w_i\), as in the proof of
Lemma~\ref{lem:matrix-reduction}.  For
\(w=(w_1,\ldots,w_k)\in W\), put
\(\dot w=(\dot w_1,\ldots,\dot w_k)\).  Also put
\(\mathcal{B}_M=\prod_i\mathcal{B}_i(F_i)\).  The product of the
matrix cell decompositions is the disjoint union
\[
M^\times
=
\bigsqcup_{w\in W}
\mathcal{U}_w\, \dot w\, \mathcal{B}_M,
\]
where
\[
\mathcal{U}_w=
\mathcal{U}_{w_1}(F_1)\times\cdots\times
\mathcal{U}_{w_k}(F_k).
\]
Here \(\mathcal{U}_{w_i}(F_i)\subset \mathcal{U}_i(F_i)\) is the
small affine space from Lemma~\ref{lem:matrix-reduction}: for
\(w_i=(w_i(1),\ldots,w_i(n_i))\), one allows a free upper-triangular
entry \(u_{w_i(b),w_i(a)}\) for each inversion \(a<b\) with
\(w_i(a)>w_i(b)\), and sets all other upper off-diagonal entries to
\(0\).  Thus
\[
\#\mathcal{U}_w=q^{\sum_i d_i\ell(w_i)}=q^{\ell_d(w)}.
\]

Decompose the inverted sum accordingly:
\begin{equation}\label{eq:ikl-cell-decomp}
\IKl_M(a;\chi)=\sum_{w\in W}\IKl_w(a;\chi),
\end{equation}
where
\[
\IKl_w(a;\chi)
=
\sum_{\substack{u^{\prime}\in \mathcal{U}_w,\, b\in \mathcal{B}_M\\
    \Nrd(u^{\prime}\dot w b)=a\\ \Trd(u^{\prime}\dot w b)\neq 0}}
\chi(b)\psi(\Trd(u^{\prime}\dot w b)^{-1}).
\]
The determinant-type character is unchanged by \(u^{\prime}\) and by the
determinant-one representative \(\dot w\).

Fix \(w\).  Since the representatives \(\dot w_i\) have determinant
\(1\), the condition \(\Nrd(u^{\prime}\dot w b)=a\) is simply
\(\Nrd(b)=a\).  Also, by cyclicity of the matrix trace in each factor,
\(\Trd(u^{\prime}\dot w b)=\Trd(bu^{\prime}\dot w)\).  For fixed
\(u^{\prime}\), the change of variables \(c=bu^{\prime}\) preserves
\(\mathcal{B}_M\), the reduced norm, and the character.  Hence
\[
\IKl_w(a;\chi)
=
\#\mathcal{U}_w
\sum_{\substack{c\in \mathcal{B}_M\\
    \Nrd(c)=a\\ \Trd(c\dot w)\neq 0}}
\chi(c)\psi(\Trd(c\dot w)^{-1}).
\]
Every upper triangular matrix \(c\in \mathcal{B}_M\) decomposes uniquely
as \(c=tu\), where \(t\in (B^{\prime})^\times\) is viewed as a diagonal
matrix and \(u\in \mathcal{U}\) is upper-unitriangular.  Here
\[
\mathcal{U}=\mathcal{U}_1(F_1)\times\cdots\times \mathcal{U}_k(F_k)
\]
is the full product of upper-unitriangular groups, an affine space of
dimension \(N=\sum_i d_i\binom{n_i}{2}\) over \(\mathbb{F}_q\).  Using
cyclicity once more, \(\Trd(tu\dot w)=\Trd(u\dot w t)\).  Substituting
\(c=tu\) gives
\begin{equation}\label{eq:ikw-expanded}
\IKl_w(a;\chi)
=
q^{\ell_d(w)}
\sum_{\substack{t\in (B^{\prime})^\times\\ \nu(t)=a}}
\eta(t)
\sum_{\substack{u\in \mathcal{U}\\
    \Trd(u\dot w t)\neq 0}}
\psi(\Trd(u\dot w t)^{-1}).
\end{equation}

When \(w=1\), \(\Trd(ut)=\Trd(t)\) for every \(u\in \mathcal{U}\).
Thus \eqref{eq:ikw-expanded} gives
\begin{equation}
\label{eq:w1}
\IKl_1(a;\chi)
=
q^N\IKl_{B^{\prime}}(a;\eta).
\end{equation}

If instead \(w\neq 1\), then some \(w_i\) is a non-identity permutation,
so its monomial representative has a nonzero entry below the diagonal.
Say this entry is in row \(s\), column \(r\), with \(s>r\).  In the
factor \(\mathcal{U}_i\), the coordinate \(u_{rs}\) is a free variable.
In this factor, put \(Y_i=\dot w_i t_i\).  Then
\[
\operatorname{Tr}(u_iY_i)
=
\operatorname{Tr}(Y_i)+\sum_{a<b}(u_i)_{ab}(Y_i)_{ba},
\]
so \(u_{rs}\) occurs in \(\operatorname{Tr}(u_i\dot w_i t_i)\) with
coefficient \((Y_i)_{sr}\neq 0\).  Here the corresponding entry of
\(\dot w_i\) is nonzero, and every diagonal entry of \(t_i\) is nonzero.

After applying the field trace
\(\Trace_{F_i/\mathbb{F}_q}\), the resulting function of the
\(F_i\)-variable \(u_{rs}\) is still nonconstant, because the pairing
\((c,z)\mapsto \Trace_{F_i/\mathbb{F}_q}(cz)\) is nondegenerate.  Thus
\[
u\longmapsto \Trd(u\dot w t)
\]
is a nonconstant affine-linear function on the \(N\)-dimensional
\(\mathbb{F}_q\)-affine space \(\mathcal{U}\).  Lemma~\ref{lem:inverted-affine}
therefore gives
\[
\sum_{\substack{u\in \mathcal{U}\\ \Trd(u\dot w t)\neq 0}}
\psi(\Trd(u\dot w t)^{-1})
=
-q^{N-1}.
\]
Substituting in \eqref{eq:ikw-expanded} yields, for \(w\neq 1\),
\begin{equation}
\label{eq:w2}
\IKl_w(a;\chi)
=
-q^{\ell_d(w)}q^{N-1}
\sum_{\substack{t\in (B^{\prime})^\times\\ \nu(t)=a}}\eta(t).
\end{equation}

Adding these contributions \eqref{eq:w1},\eqref{eq:w2} in
\eqref{eq:ikl-cell-decomp} gives
\[
\IKl_M(a;\chi)
= q^N \, \IKl_{B^{\prime}}(a;\eta)
- \sum_{\substack{w\in W\\w\neq 1}} q^{N-1+\ell_d(w)}
\sum_{\substack{t\in (B^{\prime})^\times \\ \nu(t)=a}}\eta(t).
\]
This is \eqref{eq:inverted-master}.
\end{proof}

\begin{proof}[Proof of Theorem~\ref{thm:intro-inverted-reduction}]
For the first part of the theorem, assume that \(\eta\) is nontrivial
on \(\ker(\nu)\).  For \(w\neq 1\), the fiber \(\nu^{-1}(a)\) is a
coset of the finite subgroup \(\ker(\nu)\subset (B^{\prime})^\times\).
Since \(\eta\) is nontrivial on \(\ker(\nu)\), character orthogonality
on this finite group gives
\[
\sum_{\substack{t\in (B^{\prime})^\times\\ \nu(t)=a}}\eta(t)
=
\sum_{t\in \nu^{-1}(a)}\eta(t)
=
0.
\]
Hence the second term in Proposition~\ref{prop:inverted-master}
vanishes, giving the claimed formula.

For the second part, assume that \(\eta\) is trivial on \(\ker(\nu)\), so
\(\eta=\rho\circ\nu\) for a character
\(\rho\) of \(\mathbb{F}_q^\times\).  Then every element \(t\) with
\(\nu(t)=a\) has the same \(\eta\)-value, namely \(\rho(a)\).  Hence
\[
\sum_{\nu(t)=a}\eta(t)
=
\rho(a)\#\nu^{-1}(a)
=
\rho(a)\frac{\#(B^{\prime})^\times}{q-1}.
\]
This gives the equivalent formula
\[
\IKl_M(a;\chi)
=
q^N\IKl_{B^{\prime}}(a;\eta)
-
\rho(a)\frac{\#(B^{\prime})^\times}{q-1}
\sum_{\substack{w\in W\\w\neq 1}}q^{N-1+\ell_d(w)}.
\]
Finally,
\[
\#\mathcal{B}_M=\#(B^{\prime})^\times q^N,
\qquad
\#M^\times=\#\mathcal{B}_M\sum_{w\in W}q^{\ell_d(w)}.
\]
Consequently,
\begin{align*}
  \frac{\#(B^{\prime})^\times}{q-1}
  \sum_{\substack{w\in W\\ w\neq 1}}q^{N-1+\ell_d(w)}
  &=
    \frac{\#(B^{\prime})^\times q^{N-1}}{q-1}
    \left(\sum_{w\in W}q^{\ell_d(w)}-1\right) \\
  &=
    \frac{\#\mathcal{B}_M}{q(q-1)}
    \left(\frac{\#M^\times}{\#\mathcal{B}_M}-1\right) \\
  &=
    \frac{\#M^\times-\#\mathcal{B}_M}{q(q-1)}.
\end{align*}
Substituting this into \eqref{eq:inverted-master} proves the second
formula and completes the proof.
\end{proof}

\begin{remark}
One may put these sums in a broader representation-theoretic form.  Let
\(G\) be a reductive group over \(\mathbb{F}_q\), let
\(\delta\colon G\to\mathbb{G}_{\mathrm{m}}\) be a one-dimensional
representation, and let \(\rho\colon G\to \mathrm{GL}(V)\) be another
finite-dimensional representation.  A natural Kloosterman sum attached
to this data is
\[
\Kl_{G,\rho,\delta}(a)
=
\sum_{\substack{g\in G(\mathbb{F}_q)\\ \delta(g)=a}}
\psi(\operatorname{Tr}(\rho(g))),
\qquad a\in\mathbb{F}_q^\times.
\]
One may also insert a multiplicative character of \(G(\mathbb{F}_q)\).

In this paper, the sums we consider are associated to the group
\begin{equation*}
G = \prod_{i=1}^k \mathrm{Res}_{\mathbb{F}_{q^{d_i}}} \mathrm{GL}_{n_{i}}.
\end{equation*}
Note that \(G\otimes_{\mathbb{F}_q} \overline{\mathbb{F}}_q\)
is isomorphic to a direct product
\begin{equation*}
\prod_{i=1}^k \underbrace{\mathrm{GL}_{n_{i}}\times \cdots \times \mathrm{GL}_{n_i}}_{d_i\text{ copies}}.
\end{equation*}
The character \(\delta\) is induced by the determinant morphisms, and
\(\rho\) is given by the direct sum of the standard representations of
the factors.

Such Kloosterman sums associated to reductive groups and their
representations do \emph{not} satisfy square-root cancellation in
general.  The point is that non-identity Bruhat cells may contribute.
For instance, take \(G=\mathrm{GL}_4\), let \(\delta=\det\), and let
\(\rho=\wedge^2\) be the second exterior power of the standard
representation.  If \(U\) is the upper-unitriangular subgroup,
\(t=\operatorname{diag}(t_1,t_2,t_3,t_4)\), and \(w=(12)(34)\), then
for the corresponding permutation matrix \(\dot w\) one computes
\[
\operatorname{Tr}(\wedge^2(u\dot w t))
=
-t_1t_2-t_3t_4+t_1t_3u_{12}u_{34},
\qquad u\in U.
\]
Hence
\[
\sum_{u\in U(\mathbb{F}_q)}
\psi(\operatorname{Tr}(\wedge^2(u\dot w t)))
=
q^5\psi(-t_1t_2-t_3t_4),
\]
which is not zero.

Thus the vanishing of the non-identity cells in the ordinary
Kloosterman sum considered above is a special feature of the reduced
trace on products of general linear groups, not a general phenomenon
for arbitrary representations.

On the other hand, if one allows matrix coefficients in the definition,
then, as a general feature of the \(\ell\)-adic Fourier transform,
square-root cancellation does hold if the matrix
coefficient is generic.  See \cite{fu-li_hypergeometric-reductive-groups}.
\end{remark}

\bibliographystyle{amsplain}
\bibliography{nt}

\end{document}